\documentclass[a4paper,11pt]{amsart}
\usepackage{amsfonts,amssymb,epsfig,latexsym}
\usepackage{amsmath}
\usepackage[latin1]{inputenc}
\usepackage{color,layout}
\usepackage{ifthen}

\usepackage{graphicx}
\usepackage{color}
\usepackage{bm}

\newtheorem{theorem}{Theorem}[section]
\newtheorem{lemma}[theorem]{Lemma}
\newtheorem{Prop}[theorem]{Proposition}
\newtheorem{Def}[theorem]{Definition}

\def\square{\hbox{\vrule\vbox{\hrule\phantom{o}\hrule}\vrule}}

\newcommand{\be}{\begin{equation}}
\newcommand{\ee}{\end{equation}}

\numberwithin{equation}{section}

\newcommand{\ddoto}{{\rm {\ddot{o}}}}

\numberwithin{equation}{section}

\begin{document}

\title[Eigenvalue problem of Zakharov-Shabat operator]{Real eigenvalues of a non-self-adjoint perturbation of the self-adjoint Zakharov-Shabat operator}
\author{K.~Hirota}



\begin{abstract}
We study the eigenvalues of the self-adjoint Zakharov-Shabat operator corresponding to the defocusing nonlinear Schr${\rm \ddot{o}}$dinger equation in the inverse scattering method.
Real eigenvalues exist when the square of the potential has a simple well.
We derive two types of quantization condition for the eigenvalues by using the exact WKB method, and show that the eigenvalues stay real for a sufficiently small non-self-adjoint perturbation when the potential has some $\mathcal{PT}$-like symmetry.

\end{abstract}

\maketitle
\renewcommand{\thefootnote}{\fnsymbol{footnote}}
\renewcommand{\thefootnote}{\arabic{footnote}}


\setcounter{section}{0}
\section{Introduction}
\ \ \ We consider the eigenvalue problem
\begin{align} \label{ZSeq}
L\bm{u}(x) = \lambda \bm{u}(x),
\end{align}
for the first order $2\times 2$ differential system on the line:
\begin{align*}
L:=
\begin{pmatrix}
\displaystyle ih\frac{d}{dx} &  -iA(x) \\
 iA(x) & \displaystyle -ih\frac{d}{dx}
\end{pmatrix},
\end{align*}
where $h$ is a small positive parameter, $\lambda$ is a spectral parameter, ${\bm u}(x)$ is a column vector, and $A(x)$ is a real-valued potential. This operator is called the Zakharov-Shabat operator, which is one of the two operators in the Lax pair for the defocusing  
nonlinear Schr$\ddoto$dinger equation: 
\begin{align*}
ih\frac{\partial \psi}{\partial t} + \frac{h^{2}}{2}\frac{\partial^{2} \psi}{\partial x^{2}} - \left|\psi \right|^{2} \psi = 0, \quad \psi = \psi(t,x),
\end{align*}
and the scattering theory of $L$ plays an important role in the analysis of the solutions of the initial value  
problem for this equation.

\ \ \ The operator $L$ is self-adjoint, and it is expected that $L$ has real eigenvalues when $A(x)^{2}$ has a well. In the first part of our study, we derive the Bohr-Sommerfeld type quantization condition for the eigenvalues of $L$ under the following assumption.

\ \ 
{\bf Assumption (A1).} 
Let $A(x)$ be a real-valued function analytic in $D:=\{ z\in \mathbb{C} ; \left| \rm{Im}z\right| < \delta \}$ for some $\delta >0$, and $\lambda_{0}$ a positive real number satisfying the following conditions: 
\begin{enumerate}
\item There exist two real numbers, $\alpha_{0}$ and $\beta_{0}~(\alpha_{0}<\beta_{0})$ such that $|A(x)| = \lambda_{0}, x \in \mathbb{R}$ if and only if $x = \alpha_{0},~\beta_{0}$.
\item $A^{'}(\alpha_{0})A^{'}(\beta_{0}) \neq 0.$
\item 
$|A(x)| < \lambda_{0}$ for $\alpha_{0} < x < \beta_{0},$ and $
|A(x)|> \lambda_{0}$ for $x<\alpha_{0}$ and $x>\beta_{0}$.
\item $\displaystyle \liminf_{|x| \to \infty}|A(x)| > \lambda_{0}$.
\end{enumerate}

This assumption permits two types of potentials. One is a simple well type where $A(\alpha_{0}) = A(\beta_{0})$, and the other is monotonic type where  $A(\alpha_{0})=-A(\beta_{0})$. In both cases, $A(x)^{2}$ has a simple well, see Figure \ref{fig:expotent}.

\begin{figure}[htbp]
 \begin{minipage}{0.48\hsize}
  \begin{center}
   \includegraphics[width=50mm]{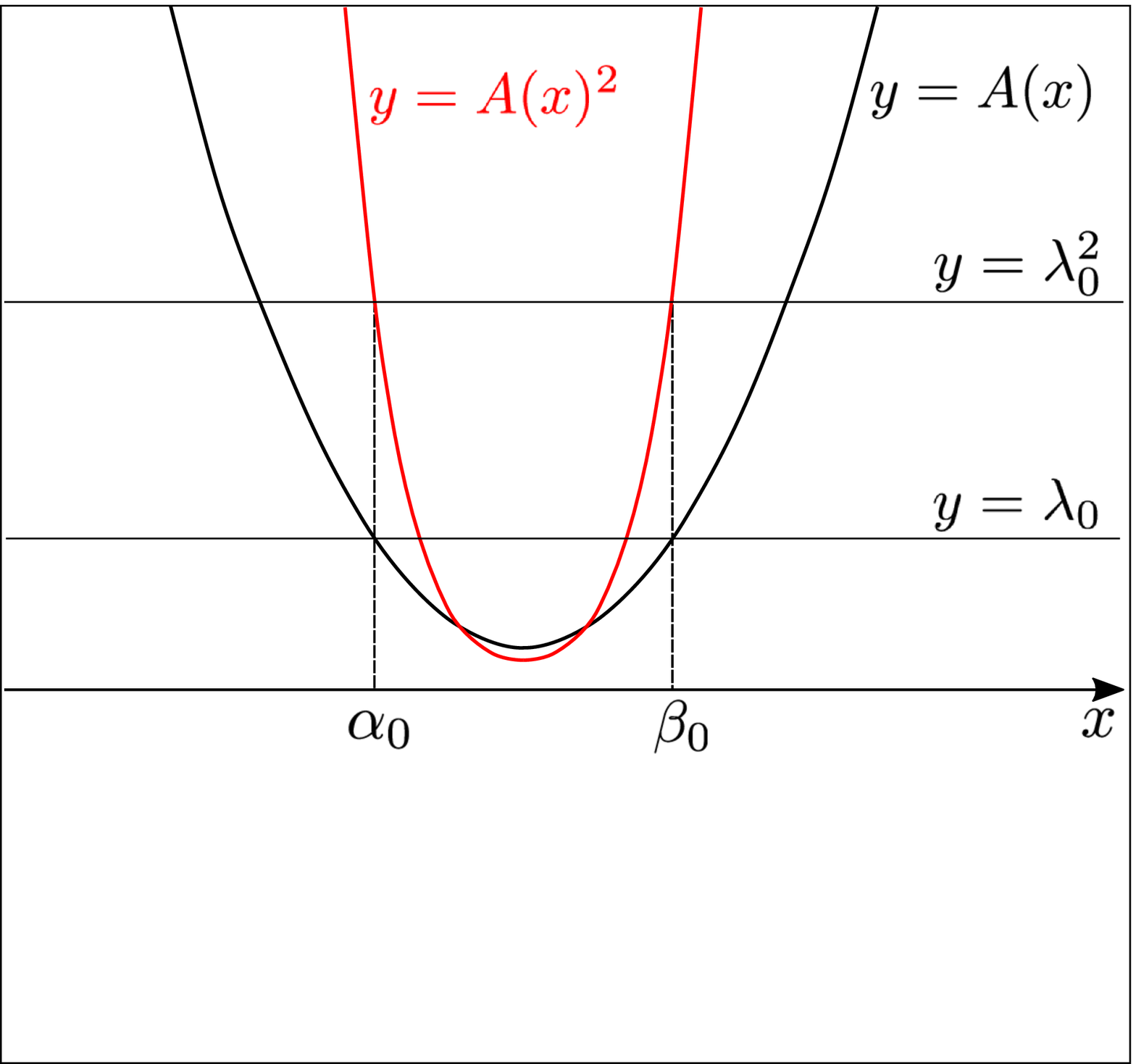}
  \end{center}
  \label{fig:one}
 \end{minipage}
 \begin{minipage}{0.48\hsize}
  \begin{center}
   \includegraphics[width=50mm]{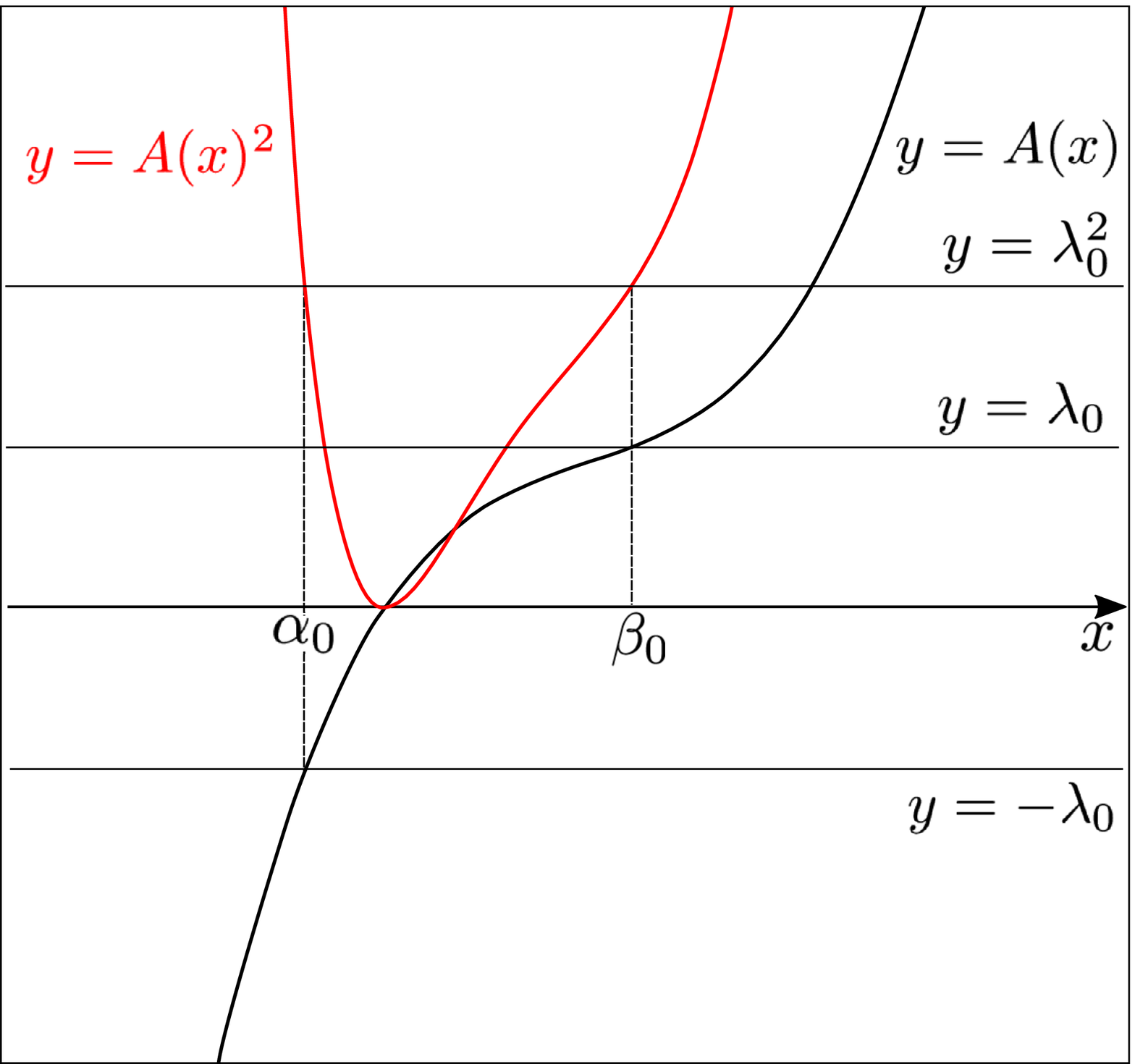}
  \end{center}
  \label{fig:expotent}
 \end{minipage}
\caption{Examples of the potential $A(x)$.}
\end{figure}

For $\lambda \in \mathbb{R}$ close enough to $\lambda_{0}$, the function $\lambda^{2}-A(x)^{2}$ has exactly two real zeros $\alpha(\lambda)$ and  $\beta(\lambda),$ close to $\alpha_{0}$ and $\beta_{0}$ respectively, and we define the action integral
\begin{align} \label{action}
I(\lambda) := \int^{\beta(\lambda)}_{\alpha(\lambda)} \sqrt{\lambda^{2}-A(t)^{2}}dt.
\end{align}
Then, we obtain the following quantization conditions. 
\begin{theorem}
Assume (A1). In the case $A(\alpha_{0}) = A(\beta_{0})$, there exist positive constants $\delta$ and $h_{0}$, and a function $r_{+}(\lambda, h)$ bounded on $[\lambda_{0}-\delta,\lambda_{0}+\delta] \times (0, h_{0}]$ such that $\lambda \in [\lambda_{0}-\delta,\lambda_{0}+\delta]$ is an eigenvalue of $L$ for $h \in (0,h_{0}]$ if and only if 
\begin{align} \label{qcon1}
I(\lambda) = \left(k+\frac{1}{2}\right)\pi h + h^{2}r_{+}(\lambda,h) 
\end{align}
holds for some integer $k$. In the case $A(\alpha_{0}) = - A(\beta_{0})$, there exist positive constants $\delta$ and $h_{0}$, and a function $r_{-}(\lambda, h)$ bounded on $[\lambda_{0}-\delta,\lambda_{0}+\delta] \times (0, h_{0}]$ such that $\lambda \in [\lambda_{0}-\delta,\lambda_{0}+\delta]$ is an eigenvalue of $L$ for $h \in (0,h_{0}]$ if and only if
\begin{align} \label{qcon2}
I(\lambda) = k\pi h + h^{2}r_{-}(\lambda,h)
\end{align}
holds for some integer $k$. 
\end{theorem}

\ \ \ Next, we add a small complex perturbation to the potential $A(x)$:
\begin{align*}
A_{\varepsilon}(x) = A(x) + i\varepsilon B(x)
\end{align*}
with a real-valued function $B(x)$ and a positive small parameter $\varepsilon$, and consider the eigenvalues of $L_{\varepsilon}$
\begin{align*}
L_{\varepsilon}:=
\begin{pmatrix}
\displaystyle ih\frac{d}{dx} &  -iA_{\varepsilon}(x) \\
 iA_{\varepsilon}(x) & \displaystyle -ih\frac{d}{dx}
\end{pmatrix}.
\end{align*}
This operator is no longer self-adjoint, and eigenvalues become complex in general. 

\ \ \ In the case of Schr${\rm \ddot{o}}$dinger operator, $\mathcal{PT}$-symmetry has been expected to be  an alternative to the self-adjointness in order to have real eigenvalues.
In recent studies, Boussekkine and Mecherout considered in for the Schr${\rm \ddot{o}}$dinger operator with  $\mathcal{PT}$-symmetry
\begin{align*}
P_{\varepsilon}:=-h^{2}\frac{d^{2}}{dx^{2}} + V(x) + i \varepsilon W(x),
\end{align*}
where $V(x)$ is a simple well even function and $W(x)$ is an odd function, and showed that reality of eigenvalues also holds for sufficiently small $\varepsilon$ and $h$. After that, Boussekkine, Mecherout, Ramond and Sj$\rm \ddot{o}$strand studied in [8] the double well case with $\mathcal{PT}$-symmetry, and found that the eigenvalues stay real  only for exponentially small $\varepsilon$ with respect to $h$.

\ \ \ In this paper, we continue in this direction and prove that a sufficiently small complex perturbation of the self-adjoint Zakharov-Shabat operator $L_{\varepsilon}$ has real eigenvalues when $A(x)$ and $B(x)$ have some $\mathcal{PT}$-like symmetry symmetry in the case where $A(x)^{2}$ has a simple well, even though the perturbed operator $L_{\varepsilon}$ is non-self-adjoint. Recalling that the condition where $P_{\varepsilon}$ is $\mathcal{PT}$-symmetric is equivalent to one where $V(x)$ is an even function and $W(x)$ is an odd function (see [4] or [8]), we assume the following symmetry properties for $A(x)$ and $B(x)$.

{\bf Assumption (A2).} 
Let $B(x)$ be real-valued, analytic and bounded on $\mathbb{R}$. $A(x)$ and $B(x)$ satisfy for $x \in \mathbb{R}$ either
\begin{align}
A(x) = A(-x), \quad B(x) = -B(-x), \label{sym1}
\intertext{or}
A(x) = -A(-x), \quad B(x) = B(-x). \label{sym2}
\end{align}

The following theorem shows that the eigenvalues of $L_{\varepsilon}$ are real for sufficiently small $\varepsilon$ and $h$.

\begin{theorem} \label{PTeigenvalue}
Assume (A1) and (A2). Then there exist positive constants $\varepsilon_{0}$ and $h_{0}$ such that $\sigma(L_{\varepsilon}) \cap \left\{\lambda \in \mathbb{C}: \left|\lambda-\lambda_{0}\right| < \varepsilon_{0} \right\} \subset \mathbb{R} $ when $0 < \varepsilon \leq \varepsilon_{0}$ and $0 < h \leq h_{0}.$
\end{theorem}

\ \ \ To prove Theorem 1.1 and 1.2, we use the exact WKB method. In Section 2, we mention the exact WKB solutions for (\ref{ZSeq}) and introduce three important properties. These exact WKB solutions are used in Section 3 to derive the quantization conditions (\ref{qcon1}) and (\ref{qcon2}). After that, we consider the perturbed case, and give the proof for Theorem 1.2 in Section 4.


\section{Exact WKB solutions}
\ \ \ We construct solutions to (\ref{ZSeq}) by the exact WKB method. This method was proposed by G\'{e}rard and Grigis in [7], and extended to $2 \times 2$ systems by Fujii\'{e}, Lasser and N\'{e}d\'{e}lec in [5]. 

\ \ \ Before the construction, we assume that $\Omega$ is a simply connected open subset of $D$, where $A(x)^{2}-\lambda^{2}$ does not vanish. Following [5], we can construct exact WKB solutions for (\ref{ZSeq}) in the form
\begin{align} \label{WKBsolution}
{\bm u}^{\pm}(x,h;\gamma,x_{0}) = 
\begin{pmatrix}
1 & 1 \\
-i & i
\end{pmatrix}
 e^{\pm z(x;\gamma)/h}Q(x)
\begin{pmatrix}
0 & 1 \\
1 & 0 
\end{pmatrix}
^{\frac{1\pm1}{2}} {\bm w}^{\pm}(x,h;x_{0})
\end{align}
with base points $\gamma \in D$ and  $x_{0} \in \Omega$. Here, $z(x;\gamma)$ is a phase function
\begin{align*}
z(x;\gamma) :=  \int_{\gamma}^{x}\sqrt{A(t)^{2}-\lambda^{2}}dt,
\end{align*}
$Q(x)$ is a $2 \times 2$ matrix function
\begin{align}
 Q(x) =
\begin{pmatrix}
H(x)^{-1} & H(x)^{-1} \\
iH(x) & -iH(x)
\end{pmatrix},\quad
H(x) = \left(\frac{A(x)+\lambda}{A(x)-\lambda}\right)^{\frac{1}{4}},
\end{align}
and ${\bm w}^{\pm}(x,h;x_{0})$ are the series
\begin{align*}
{\bm w}^{\pm}(x,h;x_{0}) = 
\begin{pmatrix}
w_{even}^{\pm}(x,h;x_{0})\\
w_{odd}^{\pm}(x,h;x_{0})
\end{pmatrix}
:=
 \sum_{n=0}^{\infty}
\begin{pmatrix}
w_{2n}^{\pm}(x,h)\\
w_{2n-1}^{\pm}(x,h)
\end{pmatrix}
\end{align*}
constructed by the recurrence equations
\renewcommand{\arraystretch}{1.7}
\begin{gather}
  \left\{
    \begin{array}{l}
      w_{-1}^{\pm} = 0,~~~w_{0}^{\pm} = 1, \\
\displaystyle
      \frac{d}{dx}w_{2n}^{\pm} =c(x)w_{2n-1}^{\pm},\qquad c(x) =\frac{H^{\prime}(x)}{H(x)},\\
\displaystyle
	\left(\frac{d}{dx} \pm \frac{2z^{\prime}}{h}\right)w_{2n-1}^{\pm} = c(x)w_{2n-2}^{\pm},
    \end{array}
  \right.
\label{diffx}
\end{gather}
\renewcommand{\arraystretch}{1.0}
and the initial conditions
\begin{align}
w_{n}^{\pm} |_{x=x_{0}}= 0 \quad (n \geq 1).  \label{init-z}
\end{align}

These solutions constructed above formally satisfy (\ref{ZSeq}). We recall here the following three propositions. The proofs are found in [5] or [7]. The first is about the convergence of series.

\begin{Prop}
Two series $w^{\pm}_{even}(x,h;x_{0})$ and $w^{\pm}_{odd}(x,h;x_{0})$ are absolutely convergent in a neighborhood of $x_{0}$. Furthermore, $w^{\pm}_{even}(x,h;x_{0})$ and $w^{\pm}_{odd}(x,h;x_{0})$ are analytic functions in $\Omega$.
\end{Prop}

The second property is about the Wronskian for two different types of exact WKB solutions.
\begin{Prop} \label{wron}
Let $\gamma, x_{0} ,x_{1} \in \Omega$ be the base points. Then, the exact WKB solutions ${\bm u}^{\pm}(x,h;\gamma,x_{0})$ and ${\bm u}^{\pm}(x,h;\gamma,x_{1})$ satisfy
\begin{align*}
\mathcal{W}\left({\bm u}^{\pm}(x,h;\gamma,x_{0}), {\bm u}^{\mp}(x,h;\gamma,x_{1})\right) = \mp 4 w_{even}^{\pm}(x_{1},h;x_{0}),
\end{align*}
where $\mathcal{W}({\bm f},{\bm g}) := \det({\bm f},{\bm g})$. This is called the Wronskian formula. 
\end{Prop}

The final proposition is about the asymptotic property of the exact WKB solution. Let $x_{0} \in \Omega $ be fixed. 
\begin{Def}
We denote by $\Omega_{\pm}$ the subset of all x $\in \Omega$ such that there exists a path in  $\Omega$ from $x_{0}$ to $x$ along which $\pm {\rm Re}z(x;x_{0})$ is strictly increasing.
\end{Def}

\begin{theorem} \label{asym}
The functions $w_{even}^{\pm}(x,h;x_{0})$ and $w_{odd}^{\pm}(x,h;x_{0})$ have the asymptotic expansions as $h \rightarrow 0$:
\begin{align*}
w_{even}^{\pm}(x,h;x_{0}) - \sum_{n=0}^{N}w_{2n}^{\pm}\left(x,h;x_{0}\right) &= \mathcal{O}(h^{N+1}),\\
w_{odd}^{\pm}(x,h;x_{0}) - \sum_{n=0}^{N}w_{2n-1}^{\pm}\left(x,h;x_{0}\right) &= \mathcal{O}(h^{N+1}),
\end{align*}
in all compact subsets of $\Omega_{\pm}$.
\end{theorem}

To find the domain $\Omega_{\pm}$, we usually consider the Stokes lines, which are level curves of the real part of $z(x;\gamma)$. In particular, the Stokes lines passing through the point $\gamma_{0} \in D$ are defined as the set
\begin{align*}
\left\{ x \in D ; {\rm Re} z(x;\gamma_{0}) = {\rm Re}  \int_{\gamma_{0}}^{x}\sqrt{A(t)^{2}-\lambda^{2}}dt = 0 \right\}.
\end{align*}
Along a path which intersects transversally with the Stokes lines, ${\rm Re}z(x)$ or $-{\rm Re}z(x)$ is strictly increasing.


\section{Quantization condition for the eigenvalues of $L$}

\ \ \ Here we find the quantization condition under Assumption (A1). This is derived from the connection problem of the solutions near the points $\alpha(\lambda)$ and $\beta(\lambda)$ which are zeros of $A(x)^{2} - \lambda^{2}$.

\ \ \ Now, we choose $\alpha(\lambda)$ and $\beta(\lambda)$ for base points of the phase function $z(x)$, and consider the Stokes lines which pass through $\alpha(\lambda)$ and $\beta(\lambda)$. By a simple calculation, we see that the Stokes lines emanate from $\alpha$ at angles of $0, 2\pi/3$ and $4\pi/3$, and emanate from $\beta$ at angles of $\pi/3 ,\pi $ and $5\pi/3$.

\begin{figure}[htbp]
 \begin{center}
  \includegraphics[width=80mm]{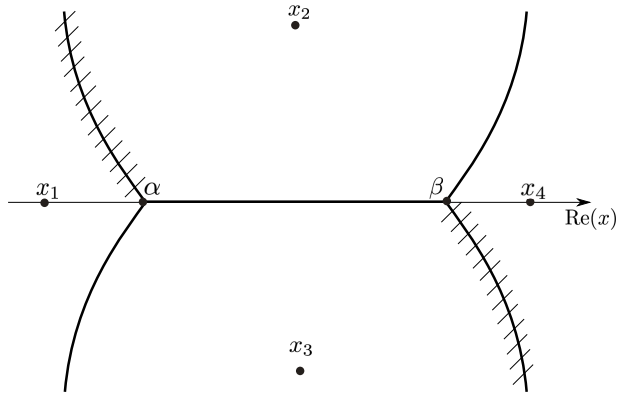}
 \end{center}
 \caption{The Stokes lines and base points}
 \label{fig:stokesone}
\end{figure}

Those Stokes lines separate the complex plane into four sectors as in Figure \ref{fig:stokesone}. As $\sqrt{A(x)^{2}-\lambda^2}$ and $H(x)$ are multi-valued functions on the complex plane with singularities at $\alpha$ and $\beta$, we set branch cuts emanating at an angle of $2\pi/3$ from $\alpha$ and an angle of $5\pi/3$ from $\beta$ respectively. We choose the branches such that ${\rm Re}\sqrt{A(x)^{2}-\lambda^2}$ and $H(x)$ are positive on a part of the real axis ${\rm Re}(x) > \beta$.

\ \ \ We take base points for $\bm{w}^{\pm}(x,h)$ in each sector as in Figure \ref{fig:stokesone}, and define the exact WKB solutions:
\begin{gather}
  \left\{
    \begin{array}{l}
      {\bm u}_{1}={\bm u}^{+}(x,h;\alpha,x_{1}) \\
{\bm u}_{2}={\bm u}^{+}(x,h;\alpha,x_{2}),~~~{\widetilde {\bm u}_{2}}={\bm u}^{+}(x,h;\beta,x_{2}),\\
{\bm u}_{3}={\bm u}^{-}(x,h;\alpha,x_{3}),~~~{\widetilde{\bm u}_{3}}={\bm u}^{-}(x,h;\beta,x_{3}),\\
{\bm u}_{4}={\bm u}^{-}(x,h;\beta,x_{4}).
    \end{array}
  \right. \label{def-wkb27}
\end{gather} 
Then, we represent ${\bm u}_{1}$ as a linear combination of $\bm{u}_{2}$ and $\bm{u}_{3}$:
\begin{align} \label{lc1}
{\bm u_{1}}=c_{2}{\bm u}_{2}+c_{3}{\bm u}_{3},
\end{align}
and ${\bm u_{4}}$ as 
\begin{align} \label{lc2}
{\bm u_{4}}=\widetilde{c}_{2}\widetilde{{\bm u}}_{2}+\widetilde{c}_{3}\widetilde{{\bm u}}_{3},
\end{align}
where each coefficient depends on $h$ and $\lambda$. We calculate those coefficients by using Theorems \ref{wron} and \ref{asym}, and obtain the following.
\begin{lemma} \label{lemmaA-1}
Assume (A1).  In the two cases $A(\alpha) = \pm A(\beta)$, the connection coefficients $c_{i}$ and $\widetilde{c_{j}}\  (i,j \in \{2,3\})$ satisfy
\begin{align*}
c_{2}{\widetilde c}_{3} = 1+\mathcal{O}(h), ~~~{\widetilde c}_{2} c_{3} = \mp 1+\mathcal{O}(h).
\end{align*}
as $h \rightarrow 0$.
\end{lemma}
\begin{proof}
Each coefficient is represented in terms of the Wronskians as
\begin{align*}
c_{2} = \frac{\mathcal{W}({\bm u}_{1},{\bm u}_{3})}{\mathcal{W}({\bm u}_{2},{\bm u}_{3})},~~~c_{3} = \frac{\mathcal{W}({\bm u}_{1},{\bm u}_{2})}{\mathcal{W}({\bm u}_{3},{\bm u}_{2})},\\
{\widetilde c}_{2} = \frac{\mathcal{W}({\bm u}_{4},{\widetilde {\bm u}_{3}})}{\mathcal{W}({\widetilde {\bm u}_{2}},{\widetilde {\bm u}_{3}})},~~~{\widetilde c}_{3} = \frac{\mathcal{W}({\bm u}_{4},{\widetilde {\bm u}_{2}})}{\mathcal{W}({\widetilde {\bm u}_{3}},{\widetilde {\bm u}_{2}})}.
\end{align*}
For $c_{2}$ and $\widetilde{c_{3}}$, we see that
\begin{align} \label{c2tilc3}
c_{2} = \frac{w_{even}^{+}(x_{3},h;x_{1})}{w_{even}^{+}(x_{3},h;x_{2})},~~\widetilde{c_{3}}=\frac{w_{even}^{+}(x_{4},h;x_{2})}{w_{even}^{+}(x_{3},h;x_{2})}.
\end{align}
by Theorem {\rm \ref{wron}}.

\ \ \ Let $\Gamma(x_{i},x_{j})$ denote a path from $x_{i}$ to $x_{j}$. We take $\Gamma(x_{1},x_{3})$, $\Gamma(x_{2},x_{3})$ and $\Gamma(x_{2},x_{4})$, and then notice that they intersect the Stokes lines, see Figure \ref{fig:stokesB}. Moreover, ${\rm Re}z(x,\cdot)$ increases as ${\rm Re}(x)$ increases, or ${\rm Im}(x)$ decreases. Therefore, ${\rm Re}z(x,\cdot)$ is strictly increasing along those paths.
\begin{figure}[htbp]
 \begin{center}
  \includegraphics[width=75mm]{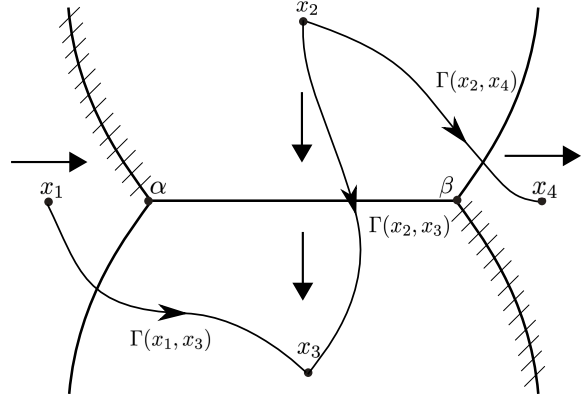}
 \end{center}
 \caption{Examples of $\Gamma(x_{i},x_{j})$. Arrows indicate directions along which ${\rm Re}z(x)$ increases.}
 \label{fig:stokesB}
\end{figure}

According to Theorem {\rm \ref{asym}}, we obtain 
\begin{align*}
w_{even}^{+}(x_{3},h;x_{1}) &= 1+\mathcal{O}(h), ~~w_{even}^{+}(x_{3},h;x_{2}) = 1+\mathcal{O}(h),\\ w_{even}^{+}(x_{4},h;x_{2}) &= 1+\mathcal{O}(h)
\end{align*}
as $h \rightarrow 0$, and we see that
\begin{align*}
c_{2} {\widetilde c}_{3}  = \frac{1+\mathcal{O}(h)}{1+\mathcal{O}(h)} \frac{1+\mathcal{O}(h)}{1+\mathcal{O}(h)} =1+\mathcal{O}(h) \qquad (h \rightarrow 0)
\end{align*}
from ({\rm \ref{c2tilc3}}). This holds in both cases $A(\alpha) = \pm A(\beta)$.

\ \ \ To  calculate the Wronskian $\mathcal{W}({\bm u}_{1},{\bm u}_{2})$, we recall that there exists a branch cut between $x_{1}$ and $x_{2}$. For this, we have to represent $\bm{u}_{1}$ or $\bm{u}_{2}$ by different branches. Let $\hat{x}$ denote a point obtained by rotating $x$ by the angle of  $-2\pi$ around $\alpha$, that is,
\begin{align*}
\hat{x}-\alpha = e^{-2\pi i}(x-\alpha).
\end{align*}

Then, we rewrite $\bm{u}_{2}$ in terms of $\hat{x}$. When $A(\alpha) = \lambda$,
\begin{align*}
\sqrt{A(x)- \lambda} = \sqrt{e^{2\pi i}\left(A(\hat{x})- \lambda\right)} = -\sqrt{\left(A(\hat{x})- \lambda\right)}.
\end{align*}
On the other hand, when $A(\alpha) = -\lambda$,
\begin{align*}
\sqrt{A(x)+ \lambda} = \sqrt{e^{2\pi i}\left(A(\hat{x})+\lambda\right)} = -\sqrt{\left(A(\hat{x})+\lambda\right)}.
\end{align*}
Therefore, there is a sign change
\begin{align}
+ z(x;\alpha) = - z(\hat{x};\alpha) \label{z}
\end{align}
in both cases  $A(\alpha) = \pm \lambda$. Since the sign of $z(x)$ changes and $c(x)=c(\hat{x})$, we find from the recurrence equation {\rm (\ref{diffx})} that 
\begin{align}
{\bm w}^{+}(x,h;x_{2})  = {\bm w}^{-}(\hat{x},h;\hat{x_{2}}). \label{w}
\end{align}

\ \ \ The representation of the function $H(x)$ is different in the cases where $A(\alpha) = \lambda$ or $A(\alpha) = -\lambda$. 
When $A(\alpha) = \lambda$,
\begin{align*}
H(x) = \left(\frac{A(x)+\lambda}{A(x)-\lambda}\right)^{\frac{1}{4}} = \left(\frac{1}{e^{2\pi i}}\frac{A(\hat{x})+\lambda}{A(\hat{x})-\lambda}\right)^{\frac{1}{4}} =  e^{-\frac{\pi i}{2}}\left(\frac{A(\hat{x})+\lambda}{A(\hat{x})-\lambda}\right)^{\frac{1}{4}} .
\end{align*}
By contrast, when $A(\alpha) = -\lambda$,
\begin{align*}
H(x) = \left(\frac{A(x)+\lambda}{A(x)-\lambda}\right)^{\frac{1}{4}} = \left(e^{2\pi i}\frac{A(\hat{x})+\lambda}{A(\hat{x})-\lambda}\right)^{\frac{1}{4}} =  e^{\frac{\pi i}{2}}\left(\frac{A(\hat{x})+\lambda}{A(\hat{x})-\lambda}\right)^{\frac{1}{4}} .
\end{align*}
That is, $H(x) = \mp iH(\hat{x})$ holds with $A(\alpha) = \pm \lambda$. In addition, this leads to
\begin{align} \label{Q}
Q(x) = \pm iQ(\hat{x})
\begin{pmatrix}
0 & 1 \\
1 & 0
\end{pmatrix}.
\end{align}

From {\rm (\ref{z}), (\ref{w}), and (\ref{Q})}, we can rewrite $\bm{u}_{2}$ as
$${\bm u_{2}} =\pm i{\bm u}^{-}(\hat{x},h;\alpha,\hat{x_{2}}),$$
 and obtain
\begin{align*}
\mathcal{W}({\bm u}_{1},{\bm u}_{2})&=\mathcal{W}\left({\bm u}^{+}(x,h;\alpha,x_{1}), \pm i {\bm u}^{-}(\hat{x},h;\alpha,\hat{x_{2}})\right) \\
&=\mp 4i\det Q w_{even}^{+}(\hat{x_{2}},h;x_{1}).
\end{align*}
in each case $A(\alpha) = \pm \lambda$.

\ \ \ In the same way, we represent $\widetilde{\bm {u}_{3}}$ by the other branch to calculate $\mathcal{W}({\bm u}_{4},\widetilde{{\bm u}}_{3})$. Let $\widetilde{x}$ denote a point obtained by rotating $x$ by the angle of  $-2\pi$ around $\beta$. When $A(\beta) = \pm \lambda$, $\widetilde{\bm{u}_{3}}$ is rewritten as
\begin{align*}
\widetilde{{\bm u}_{3}}=\pm i{\bm u}^{+}(\widetilde{x},h;\beta,\widetilde{x_{3}}).
\end{align*}
Therefore, $\mathcal{W}({\bm u}_{4},\widetilde{{\bm u}}_{3})$ is calculated as
\begin{align*}
\mathcal{W}({\bm u}_{4},\widetilde{{\bm u}}_{3})
&=-\mathcal{W}\left(\pm i{\bm u}^{+}(x,h;\beta,\widetilde{x_{3}}),{\bm u}^{-}(x,h;\beta,{x_{4}})\right)\\
&=\pm 4i\det Q w_{even}^{+}(x_{4},h;\widetilde{x_{3}}).
\end{align*}

\ \ \ We can find $\Gamma(x_{1},\hat{x_{2}})$ and  $\Gamma(\widetilde{x_{3}},x_{4})$ along which ${\rm Re}z(x,\cdot)$ strictly increasing, and obtain
\begin{align*}
w_{even}^{+}(\hat{x_{2}},h;x_{1}) = 1+\mathcal{O}(h), \quad w_{even}^{+}(x_{4},h;\widetilde{x_{3}}) = 1+\mathcal{O}(h)
\end{align*}
as $h \rightarrow 0$.

\ \ \ As a result, we obtain that when $A(\alpha)A(\beta)>0$, 
\begin{align*}
c_{3}\widetilde{c_{2}}  &=i^{2} \frac{w_{even}^{+}(\hat{x_{2}},h;x_{1})}{w_{even}^{+}(x_{3},h;x_{2})}\frac{ w_{even}^{+}(x_{4},h;\widetilde{x_{3}})}{w_{even}^{+}(x_{3},h;x_{2})} = - 1+ \mathcal{O}(h)
\intertext{as $h \rightarrow 0$. In the case $A(\alpha)A(\beta)<0$,}
c_{3}\widetilde{c_{2}}  &= -i^{2}\frac{w_{even}^{+}(\hat{x_{2}},h;x_{1})}{w_{even}^{+}(x_{3},h;x_{2})}\frac{w_{even}^{+}(x_{4},h;\widetilde{x_{3}})}{w_{even}^{+}(x_{3},h;x_{2})} = 1+ \mathcal{O}(h)
\end{align*}
as $h \rightarrow 0$.
\end{proof}

\ \ \ Here we return to equation (\ref{ZSeq}). The spectral parameter $\lambda$ near $\lambda_{0}$ is an eigenvalue of $L$ if and only if $\bm{u}_{1}$ and $\bm{u}_{4}$ are linearly dependent, since $\bm{u}_{1} \in L^{2}(\mathbb{R}_{-})$ and $\bm{u}_{4} \in L^{2}(\mathbb{R}_{+})$. That is, we consider the condition 
\begin{align}
\mathcal{W}({\bm u}_{1},{\bm u}_{4})=0. \label{quant2}
\end{align}
From (\ref{lc1}) and (\ref{lc2}), we know that the Wronskian $\mathcal{W}({\bm u}_{1},{\bm u}_{4})$ is expressed in terms of $\bm{u}_{j}$ and $\widetilde{\bm{u}_{k}}$ $(j ,k \in \left\{2,3\right\})$ as
\begin{align*}
\mathcal{W}({\bm u}_{1},{\bm u}_{4}) = c_{2}{\widetilde c}_{3}\mathcal{W}({\bm u}_{2}, {\widetilde {\bm u}_{3}})-{\widetilde c}_{2}c_{3}\mathcal{W}({\widetilde {\bm u}_{2}},{\bm u}_{3}). 
\end{align*}
Since $\bm{u}_{j}$ and $\widetilde{\bm{u}_{j}}$ are linearly dependent and satisfy
\begin{align*}
{\widetilde {\bm u}_{2}}=e^{-iI(\lambda)/h}  {\bm u_{2}},~
{\widetilde {\bm u}_{3}}=e^{iI(\lambda)/h}  {\bm u_{3}},
\end{align*}
condition (\ref{quant2}) is equivalent to
\begin{align*}
c_{2}{\widetilde c}_{3}e^{iI(\lambda)/h}-{\widetilde c}_{2}c_{3}e^{-iI(\lambda)/h}=0.
\end{align*}
That is, $I(\lambda)$ satisfies
\begin{align*}
I(\lambda) + \frac{h}{2i} \log \left(-\frac{c_{2}{\widetilde c}_{3}}{{\widetilde c}_{2}c_{3}}\right)= \left(k+\frac{1}{2}\right)\pi h
\end{align*}
for some integer $k$.

\ \ \ From Lemma  \ref{lemmaA-1}, when $A(\alpha)A(\beta)>0$,
\begin{align*}
\hspace{1.5cm} \log \left(-\frac{c_{2}{\widetilde c}_{3}}{{\widetilde c}_{2}c_{3}}\right) &= \log\left(1+\mathcal{O}(h)\right) = \mathcal{O}(h)  \qquad (h \rightarrow 0).
\intertext{On the other hand, when $A(\alpha)A(\beta)<0$,}
\log \left(-\frac{c_{2}{\widetilde c}_{3}}{{\widetilde c}_{2}c_{3}}\right) &= \log\left(-1+\mathcal{O}(h)\right) =
\pi i + \mathcal{O}(h)  \qquad (h \rightarrow 0).
\end{align*}
In conclusion, the quantization condition for eigenvalues $\lambda$ is given by
\begin{align*}
\hspace{2.5cm} I(\lambda)  &= \left(k+\frac{1}{2}\right)\pi h + \mathcal{O}(h^{2}) \qquad (h \rightarrow 0),
\intertext{in the case $A(\alpha)A(\beta)>0$. Similarly, we obtain}
I(\lambda) &= k\pi h + \mathcal{O}(h^2) \qquad (h \rightarrow 0),
\end{align*}
in the case $A(\alpha)A(\beta)<0$.

\section{Eigenvalue problem for the non-self-adjoint case}
\ \ \ In this section, we consider the eigenvalue problem:
\begin{align} \label{ZSeq2}
L_{\varepsilon}\bm{u}(x) = \lambda \bm{u}(x), \quad
L_{\varepsilon}:=
\begin{pmatrix}
\displaystyle ih\frac{d}{dx} &  -iA_{\varepsilon}(x) \\
 iA_{\varepsilon}(x) & \displaystyle -ih\frac{d}{dx}
\end{pmatrix},
\end{align}
for $A_{\varepsilon}(x) = A(x) + i \varepsilon B(x)$ with $\varepsilon > 0$. First we consider the quantization condition for the eigenvalues. Here we assume that $A(x)$ satisfies Assumption (A1) and $B(x)$ is real-valued, analytic and bounded on $\mathbb{R}$. 

\ \ \ Let $D(\lambda_{0},\varepsilon_{0}) = \left\{ x \in \mathbb{C};  |x-\lambda_{0}|<\varepsilon_{0}\right\}$ for a positive $\varepsilon_{0}$.  Under Assumption (A1), for all $\lambda \in D(\lambda_{0},\varepsilon_{0})$ and $\varepsilon \in (0,\varepsilon_{0}]$, there exist zeros of $A_{\varepsilon}(x)^{2}-\lambda^{2}$, $\alpha(\lambda,\varepsilon)$ and $\beta(\lambda,\varepsilon)$ such that $\alpha(\lambda_{0},0) = \alpha_{0}$ and $\beta(\lambda_{0},0) = \beta_{0}$. We simply write them as $\alpha_{\varepsilon}$ and $\beta_{\varepsilon}$, and define the action integral $I(\lambda,\varepsilon)$:
\begin{eqnarray} 
I(\lambda,\varepsilon)=\int^{\beta_{\varepsilon}(\lambda)}_{\alpha_{\varepsilon}(\lambda)} \sqrt{A_{\varepsilon}(t)^{2}-\lambda^{2}}dt. \label{act3}
\end{eqnarray}
In addition,  the exact WKB solutions for (\ref{ZSeq2}) are given by replacing  $A(x)$ with $A_{\varepsilon}(x)$ in (\ref{WKBsolution}), and we denote those solutions by ${\bm u}^{\pm}(x,h,\varepsilon;\gamma,x_{0})$.

\ \ \ We choose $\alpha_{\varepsilon}$ and $\beta_{\varepsilon}$ for the base points of the phase function $z(x,\varepsilon)$.
The Stokes lines which pass through the points $\alpha_{\varepsilon}$ and $\beta_{\varepsilon}$ are drown in Figure \ref{fig:stokestwo}.
\begin{figure}[htbp]
 \begin{center}
  \includegraphics[width=80mm]{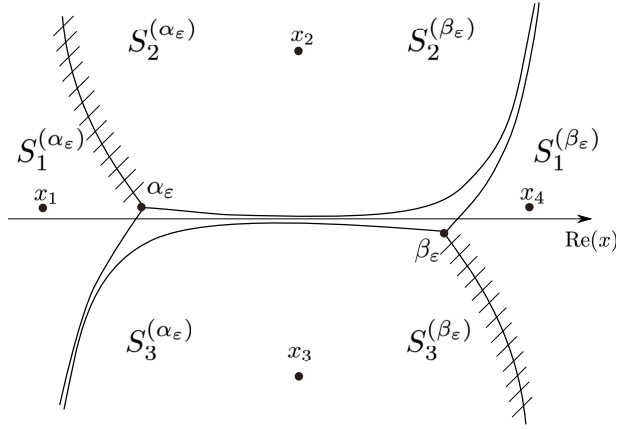}
 \end{center}
 \caption{The Stokes lines for a sufficiently small $\varepsilon$. $S_{j}^{(\alpha_{\varepsilon})}$ and $S_{k}^{(\beta_{\varepsilon})}$ indicate the sector which is generated by the Stokes lines emanating from $\alpha_{\varepsilon}$ and $\beta_{\varepsilon}$ respectively.}
 \label{fig:stokestwo}
\end{figure}

\ \ \ The Stokes lines continuously change with respect to $\varepsilon$ from the case of $\varepsilon=0$, since $\alpha_{\varepsilon}$, $\beta_{\varepsilon}$ and $z(x,\varepsilon)$ are continuous with respect to $\varepsilon$.
Here, we assume that $\varepsilon$ is sufficiently small, and take  base points as in Figure \ref{fig:stokestwo}.
Then, we can derive the quantization conditions for eigenvalues of $L_{\varepsilon}$ in the same way as the previous section.
\begin{lemma}
Assume (A1), and let $B(x)$ be real-valued, analytic and bounded on $\mathbb{R}$. In the case $A(\alpha_{0}) = A(\beta_{0})$, there exist positive constants $\varepsilon_{0}$ and $h_{0}$, and a function $r_{+}(\lambda,\varepsilon,h)$ bounded on $D(\lambda_{0},\varepsilon_{0}) \times (0, \varepsilon_{0}]  \times (0, h_{0}] $ such that $\lambda \in D(\lambda_{0},\varepsilon_{0})$ is an eigenvalue of $L_{\varepsilon}$ for $\varepsilon \in (0,\varepsilon_{0}]$ and $h \in (0,h_{0}]$ if and only if 
\begin{align}
I(\lambda,\varepsilon) = \left(k+\frac{1}{2}\right)\pi h + h^{2}r_{+}(\lambda,\varepsilon,h)
\end{align}
holds for some integer $k$. In the case $A(\alpha_{0}) = - A(\beta_{0})$, there exist positive constants $\varepsilon_{0}$ and $h_{0}$, and a function $r_{-}(\lambda,\varepsilon,h)$ bounded on $D(\lambda_{0},\varepsilon_{0}) \times (0, \varepsilon_{0}]  \times (0, h_{0}] $ such that $\lambda \in D(\lambda_{0},\varepsilon_{0})$ is an eigenvalue of $L_{\varepsilon}$ for $\varepsilon \in (0,\varepsilon_{0}]$ and $h \in (0,h_{0}]$ if and only if 
\begin{align}
I(\lambda,\varepsilon) = k\pi h + h^{2}r(\lambda,\varepsilon,h)
\end{align}
holds for some integer $k$. 
\end{lemma}

\ \ \ Now, we assume Assumption (A2) for $A_{\varepsilon}(x)$,  which results in a symmetry
of the action integral $I(\lambda,\varepsilon)$ and the exact WKB solutions ${\bm u}^{\pm}(x,h,\varepsilon;\gamma,x_{0})$ with respect to complex conjugation.
\begin{lemma}\label{acsym}
Under Assumption (A2), the action integral $I(\lambda,\varepsilon)$ is equal to the complex conjugate of $I(\overline{\lambda},\varepsilon)$:
\begin{align*}
\overline{I(\overline{\lambda},\varepsilon)} = I(\lambda,\varepsilon).
\end{align*}
\end{lemma}
\begin{proof}
By a simple calculation, we find that $-\overline{\beta_{\varepsilon}}$ and $-\overline{\alpha_{\varepsilon}}$ are zeros of $A_{\varepsilon}(x)^{2} - \overline{\lambda}^{2}$ under Assumption (A2). 
That is, $I(\overline{\lambda},\varepsilon)$ is represented as
\begin{align*}
I(\overline{\lambda},\varepsilon)=\int_{-\overline{\beta_{\varepsilon}}}^{-\overline{\alpha_{\varepsilon}}} \sqrt{\overline{\lambda}^{2}-A_{\varepsilon}(t)^{2}} dt.
\end{align*}
We take the complex conjugate of this, and obtain that
\begin{align*}
\overline{I(\overline{\lambda}, \varepsilon)}
&=\int^{-\overline{\alpha_{\varepsilon}}}_{-\overline{\beta}_{\varepsilon}} \sqrt{ \overline{\left( \overline{\lambda}^2-A_{\varepsilon}(t)^{2}\right)}}d\overline{t}
=\int^{-\overline{\alpha_{\varepsilon}}}_{-\overline{\beta}_{\varepsilon}} \sqrt{  \lambda^2-\overline{A_{\varepsilon}(t)}^{2}}d\overline{t}.
\intertext{Then, we change the variable from $t$ to $-\overline{t}$,}
\overline{I(\overline{\lambda}, \varepsilon)}&=-\int^{\alpha_{\varepsilon}}_{\beta_{\varepsilon}} \sqrt{\lambda^2-\overline{A_{\varepsilon}(-\overline{t})}^{2}}dt
=\int_{\alpha_{\varepsilon}}^{\beta_{\varepsilon}} \sqrt{  \lambda^2-A_{\varepsilon}(t)^{2}}dt. 
\end{align*}
This is just the action integral.
\end{proof}

\ \ \ We denote the exact WKB solutions for the equation
\begin{align} \label{ZSeq3}
L_{\varepsilon}{\bm v}(x) &= \overline{\lambda} {\bm v}(x)
\end{align}
by ${\bm v}^{\pm}$. Then, ${\bm v}^{\pm}$ is obtained by replacing $\lambda$ with $\overline{\lambda}$. Under Assumtion (A2),  ${\bm u}^{\pm}$ and ${\bm v}^{\pm}$ also have the following symmetry relations.

\begin{lemma}  \label{WKBsym}
Under Assumption (A2), if $\overline{A_{\varepsilon}(-\overline{x})}=A_{\varepsilon}(x)$, the exact WKB solutions ${\bm u}^{\pm}(x,h,\varepsilon;\gamma,x_{0})$ and ${\bm v}^{\pm}(x,h,\varepsilon;\gamma,x_{0})$ satisfy
\begin{align*}
{\bm u}^{\pm}(x,h,\varepsilon;\gamma,x_{0})= 
\begin{pmatrix}
1 & 0 \\
0 & -1 
\end{pmatrix}
\overline{{\bm v}^{\mp}(-\overline{x},h,\varepsilon;-\overline{\gamma},-\overline{x_{0}})},
\end{align*}
and if $\overline{A_{\varepsilon}(-\overline{x})}=-A_{\varepsilon}(x)$, then
\begin{align*}
{\bm u}^{\pm}(x,h,\varepsilon;\gamma,x_{0}) = \pm
\begin{pmatrix}
0 & -1 \\
1 & 0 
\end{pmatrix}
\overline{{\bm v}^{\mp}(-\overline{x},h,\varepsilon;-\overline{\gamma},-\overline{x_{0}})}.
\end{align*}
\end{lemma}
\begin{proof}
Let $\overline{A_{\varepsilon}(-\overline{x})}=A_{\varepsilon}(x) $. By taking the complex conjugate and changing the variable $x$ to $-\overline{x}$ for the functions $z$ and $\bm{w}^{\pm}$ of the solutions $\bm{v}^{\pm}$, we obtain 
\begin{align*}
\overline{z(-\overline{x},\varepsilon;-\overline{\gamma};\overline{\lambda})} &= -z(x,\varepsilon;\gamma;\lambda),\\
\intertext{and}
\overline{{\bm w}^{\pm}(-\overline{x},h,\varepsilon;-\overline{x_{0}};\overline{\lambda})} &= {\bm w}^{\mp}(x,h,\varepsilon;x_{0};\lambda).
\end{align*}
In the same way, for the matrix function $Q(x,\varepsilon)$ of $\bm{v}^{\pm}$, 
\begin{align*}
\overline{Q(-\overline{x},\varepsilon;\overline{\lambda})} = 
Q(x,\varepsilon;\lambda)
\begin{pmatrix}
0 & 1 \\
1 & 0
\end{pmatrix}.
\end{align*}
Here, we recall that the solutions $\bm{v}^{\pm}$ is of the form
\begin{align*}
{\bm v}^{\pm}(x,h,\varepsilon;\gamma,x_{0}) = 
\begin{pmatrix}
1 & 1 \\
-i & i
\end{pmatrix}
 e^{\pm z(x,\varepsilon;\gamma)/h}Q(x,\varepsilon)
\begin{pmatrix}
0 & 1 \\
1 & 0 
\end{pmatrix}
^{\frac{1\pm1}{2}} {\bm w}^{\pm}(x,h,\varepsilon;x_{0}).
\end{align*}
By taking the complex conjugate and changing the variable $x$ to $-\overline{x}$, and using above, we obtain the first relation
\begin{align*}
\begin{pmatrix}
1 & 0 \\
0 & -1
\end{pmatrix}
\overline{{\bm v}^{\mp}(-\overline{x},h,\varepsilon;-\overline{\gamma},-\overline{x_{0}})}=
{\bm u}^{\pm}(x,h,\varepsilon;\gamma,x_{0}).
\end{align*}

If $\overline{A_{\varepsilon}(-\overline{x})}=-A_{\varepsilon}(x) $, then we find that
\begin{align*}
\overline{z(-\overline{x},\varepsilon;-\overline{\gamma};\overline{\lambda})} &= -z(x,\varepsilon;\gamma;\lambda),\\
\overline{{\bm w}^{\pm}(-\overline{x},h,\varepsilon;-\overline{x_{0}};\overline{\lambda})} &=
\begin{pmatrix}
1 & 0 \\
0 & -1
\end{pmatrix}
{\bm w}^{\mp}(x,h,\varepsilon;x_{0};\lambda),
\end{align*}
and
\begin{align*}
\overline{Q(-\overline{x},\varepsilon;\overline{\lambda})} =  -i
\begin{pmatrix}
0 & 1 \\
-1 & 0
\end{pmatrix}
Q(x,\varepsilon;\lambda)
\begin{pmatrix}
0 & 1 \\
-1 & 0
\end{pmatrix}.
\end{align*}
From this property, the second relation also follows.
\end{proof}

\ \ \ Here we take the base points $x_{1} \in S^{(\alpha_{\varepsilon})}_{1}$ and $x_{2} \in S^{(\alpha_{\varepsilon})}_{2}$ so that $-\overline{x_{1}} \in S^{(\beta_{\varepsilon})}_{1}$ and $-\overline{x_{2}} \in S^{(\beta_{\varepsilon})}_{3}$, and set the exact WKB solutions for (\ref{ZSeq2}):
\begin{gather*}
  \left\{
    \begin{array}{l}
      {\bm u}_{1}={\bm u}^{+}(x,h, \varepsilon ;\alpha_{\varepsilon},x_{1}), \quad {\bm u}_{2}={\bm u}^{+}(x,h,\varepsilon;\alpha_{\varepsilon},x_{2}),\\
{\bm u}_{3}={\bm u}^{-}(x,h,\varepsilon;\alpha_{\varepsilon},-\overline{x_{2}})\quad
{\bm u}_{4}={\bm u}^{-}(x,h,\varepsilon;\beta_{\varepsilon},-\overline{x_{1}}).
    \end{array}
  \right. 
\end{gather*}
In addition, let us define a function $W(\lambda, \varepsilon)$ by the Wronskian of  ${\bm u}_{1}$ and ${\bm u}_{4}$, that is,
\begin{align*}
W(\lambda,\varepsilon) := \mathcal{W}( {\bm u}_{1}, {\bm u}_{4}).
\end{align*}
We also take the solutions for (\ref{ZSeq3}) as
\begin{gather*}
  \left\{
    \begin{array}{l}
      {\bm v}_{1}={\bm v}^{+}(x,h, \varepsilon ;-\overline{\beta_{\varepsilon}},x_{1}), \quad {\bm v}_{2}={\bm v}^{+}(x,h,\varepsilon;-\overline{\alpha_{\varepsilon}},x_{2}),\\
{\bm v}_{3}={\bm v}^{-}(x,h,\varepsilon;-\overline{\alpha_{\varepsilon}},-\overline{x_{2}})\quad
{\bm v}_{4}={\bm v}^{-}(x,h,\varepsilon;-\overline{\alpha_{\varepsilon}},-\overline{x_{1}}).
    \end{array}
  \right. 
\end{gather*}
Then, we see that 
\begin{align} \label{Wsym}
W(\lambda,\varepsilon) = \pm \overline{W(\overline{\lambda}, \varepsilon)},
\end{align}
by the definition of $W(\lambda,\varepsilon)$ and applying Lemma \ref{WKBsym}.
Here, the sign of (\ref{Wsym}) is dependent on whether $A_{\varepsilon}(x) = \overline{A(-\overline{x})}$ or $A_{\varepsilon}(x) = -\overline{A(-\overline{x})}$.

\ \ \ Now, we recall that $W(\lambda,\varepsilon)$ is represented as
\begin{align} \label{W1}
W(\lambda,\varepsilon)= a(\lambda,\varepsilon,h)e^{iI(\lambda,\varepsilon)/h} + b(\lambda,\varepsilon,h)e^{-iI(\lambda,\varepsilon)/h},
\end{align}
where $a$ and $b$ are some functions with $a = 1+\mathcal{O}(h)$ and  $b = 1+\mathcal{O}(h)$ or $-1+\mathcal{O}(h)$ as $h \rightarrow 0 $. In particular, $a$ and $b$ satisfy
\begin{align} \label{absym}
a(\lambda,\varepsilon,h) = \pm \overline{b(\overline{\lambda},\varepsilon,h)},
\end{align}
since $W(\lambda,\varepsilon)$ satisfy (\ref{Wsym}), and $I(\lambda,\varepsilon)$ also satisfy Lemma \ref{acsym}.
Moreover, (\ref{W1}) is rewritten as
\begin{align*}
W(\lambda,\varepsilon)= b(\lambda,\varepsilon,h)e^{-I(\lambda,\varepsilon)/h}\left(\exp \left(\frac{2i}{h}\left(I(\lambda,\varepsilon) +h^{2} r(\lambda,\varepsilon,h)\right)\right) +1 \right),
\end{align*}
where
\begin{align*}
r(\lambda,\varepsilon,h) = \frac{1}{2ih}\log\frac{a(\lambda,\varepsilon,h)}{b(\lambda,\varepsilon,h)}.
\end{align*}
Then, we use (\ref{absym}) and obtain that 
\begin{align} \label{rsym}
r(\lambda,\varepsilon,h) = \overline{r(\overline{\lambda},\varepsilon,h)}.
\end{align}

Here we take $I(\lambda,\varepsilon,h)$ as
\begin{align*}
I(\lambda,\varepsilon,h):=I(\lambda,\varepsilon) +h^{2} r(\lambda,\varepsilon,h).
\end{align*}
This is a function from a neighborhood of $\lambda_{0}$ to one of $I(\lambda_{0}, 0)$. In particular, $r(\lambda,\varepsilon,h)$ is holomorphic near $\lambda_{0}$, and $I(\lambda,\varepsilon)$ satisfies $\frac{dI}{d\lambda}(\lambda_{0},0) \neq 0$. 
This implies that $I(\lambda,\varepsilon,h)$ has an inverse function $I^{-1}(\zeta,\varepsilon,h)$, and the eigenvalues of $L_{\varepsilon}$ near $\lambda_{0}$ are given by
\begin{align*}
\lambda_{\varepsilon,k} = I^{-1}\left(c_{k} \pi h, \varepsilon, h \right), \quad k \in \mathbb{Z},
\end{align*}
where $c_{k} = k$ or $k + 1/2$. In addition, we know that $I(\lambda,\varepsilon,h)$ is real for $\lambda \in \mathbb{R}$, since $r(\lambda,\varepsilon,h)$ and $I(\lambda,\varepsilon)$ is real for $\lambda \in \mathbb{R}$ by (\ref{rsym}) and Lemma \ref{acsym}. That is, the eigenvalues near $\lambda_{0}$ are real.

\end{document}